\documentclass{amsart}
\usepackage{amsthm,graphicx,amssymb}
\usepackage[usenames]{color} 
\usepackage{colortbl}
\usepackage{caption}
\usepackage{subcaption}
\usepackage{hyperref}
\usepackage{tikz-cd}
\usepackage{lineno}
\usepackage{tikz}
\usepackage{tikz-qtree}
\usetikzlibrary{arrows}
\usetikzlibrary{calc}

\theoremstyle{definition}
 \newtheorem{theorem}{Theorem}
 
 \newtheorem{definition}[theorem]{Definition}
 \newtheorem{remark}[theorem]{Remark}
 \newtheorem{lemma}[theorem]{Lemma}

\newcommand\RR{\mathbb R}

\newcommand\ZZ{\mathbb Z}

\newcommand\G{\Gamma}
\newcommand\val{\mathrm{val}}

\begin{document}

\title{Sandpile group of infinite graphs}
\author{Nikita Kalinin, Vladislav Khramov}
\thanks{Corresponding author: Nikita Kalinin, nikita.kalinin@gtiit.edu.cn}
\address{Nikita Kalinin: Guangdong Technion Israel
Institute of Technology (GTIIT), Technion-Israel Institute of Technology, address: Guangdong Technion Israel Institute of Technology,
241 Daxue Road, Shantou, Guangdong Province 515603, P.R. China, nikita.kalinin@gtiit.edu.cn, Vladislav Khramov: Saint Petersburg State University, 7/9, Unversitetskaya emb., 199034, Saint Petersburg, Russia}
\begin{abstract}
For a finite connected graph $G$ and a non-empty subset $S$ of its vertices thought of sinks, the so-called critical group (or sandpile group) $C(G, S)$ has been studied for a long time. We present a class of graphs where such an extension can be made in a unified way. Similar extension was made by Maes, C. and Redig, F. and Saada, E., but we propose a more algebraic point of view.
  
Namely, consider a $C$-net $S\subset \mathbb Z^2$. We define a sandpile dynamics on $\mathbb Z^2$ with the set $S$ of sinks. For such a choice of sinks, a relaxation of any bounded state is well defined. This allows us to define a group $C(\mathbb Z^2, S)$ of recurrent states of this model. We show that $C(\mathbb Z^2, S)$ is isomorphic to a group of $S^1$-valued discrete harmonic functions on $\mathbb Z^2\setminus S$. 

Examples of $S$, for which $C(\mathbb Z^2, S)$ has no torsion or has all torsions, are provided. Pontryagin dual point of view is investigated. A discussion about perspectives of a sandpile group for $\mathbb Z^2$ as a projective limit concludes this work.

MCS codes: 05C63, 20E18, 05C10, 05C25
\end{abstract}

\maketitle

For a finite connected graph $G$, its critical group $K(G)$ is a finite abelian group whose order equals the number of the spanning trees of $G$. Moreover, there exists a (non-canonical) bijection \cite{Dhar} between the spanning trees of $G$ and the elements of $K(G)$; these elements can also be treated as recurrent states of the sandpile dynamics on $G$.  

It is a natural question of which connections survive when a graph grows and becomes infinite. For example, much is known about the set of uniform spanning trees (forests) in $\ZZ^2$ \cite{pemantle1991choosing,jarai2011abelian}. If the uniform spanning tree almost surely has one end, then a map from spanning trees to the recurrent sandpile states can be defined even in the infinite case \cite{gamlin2014anchored}. As we mentioned above, for a finite graph, the set of its spanning trees can be provided with the structure of an abelian group. How to keep it while passing to an infinite graph is unclear. Another motivation to study sandpiles on infinite graphs is that some of the sandpiles' properties can be formulated only via a limiting procedure, e.g., the famous self-organized criticality property states that a sandpile has a power law dependence between the size of an avalanche and its frequency \cite{bhupatiraju2016inequalities}.  

The first section reviews the key facts about a sandpile group for a finite graph. In the second section, we explain what should be changed to extend the definition of a sandpile group to infinite graphs. As we learned recently, our approach is quite parallel to that of Maes, C. and Redig, F. and Saada, E. \cite{MR2031036}. However, we hope that our presentation provides another point of view.  In the third section, we present proofs. The fourth section is dedicated to the simplest non-trivial example. In the fifth and sixth sections, we study the torsion of the sandpile group for infinite subgraphs of $\ZZ^2$ (our main object of study).  In the seventh section, we discuss the Pontyagin duality perspective. The last section contains a picture of the relaxation of a large pile of sand on a periodically punctured plane and a discussion about that elusive sandpile group on $\ZZ^2$. 

We exploit two key simple ideas: 1) if a set $S\subset\ZZ^2$ is a $C$-net, then there exists an explicit positive integer-valued function on $\ZZ^2$ which is strictly superharmonic on $\ZZ^2\setminus S$, which, in turn, guarantees that any bounded sandpile state can be relaxed, and so we can define a sandpile group $C(\ZZ^2, S)$ as in the finite case; 2) using the same function we show that the discrete Laplace operator on real-valued bounded function on $\ZZ^2\setminus S$ is invertible, which, by the snake lemma, provides an isomorphism between the group of $S^1$-valued harmonic functions on $\ZZ^2\setminus S$ and $C(\ZZ^2, S)$. This isomorphism simplifies studying the torsion of $C(\ZZ^2, S)$. The group of $S^1$-valued harmonic functions on $\ZZ^2$ (so-called ``harmonic model'') and its connection to sandpile has been extensively studied in \cite{schmidt2009abelian, shirai2016solvable}. We would like to thank Mikhail Shkolnikov, Evgeny Verbitskiy, Fedor Petrov for fruitful discussions.

\section{Sandpile group for finite graphs}
Consider a finite connected graph $G$ with a non-empty subset $S$ of its vertices. The following ``sandpile'' model was introduced in \cite{BTW} as an archetypical example of self-organized criticality. A {\it sandpile state} $\psi$ is a function $\psi:G\to \ZZ_{\geq 0}$ on the vertices of $G$; $\psi(z)$ represents the number of sand grains at $z$. If the number of grains at $z$ is at least its valency ($\psi(z)\geq \val(z)$) and $z\notin S$, then $\psi$ is called {\it unstable} at $z$ and one can perform a {\it toppling} at $z$ by sending $\val(z)$ grains from $z$ to its neighbors $z'\sim z$, one grain to each neighbor. Doing toppling while possible is called a {\it relaxation}. Any relaxation terminates (because topplings at $S$ are not allowed), and the final stable state is denoted by $\psi^{\circ}$; this state does not depend on a particular relaxation\cite{MR2246566,creutz1991abelian}.

For a given state $\psi$ a {\it toppling function} $F: G\to \ZZ_{\geq 0}$ is defined, $F(z)$ is equal to the number of topplings at a vertex $z$ during a relaxation. Naturally, $F_{|S}=0$. Note that $\psi^\circ=\psi+\Delta F$ where 

$$\Delta F(z)=\sum_{z'\sim z} F(z')-\val(z)\cdot F(z).$$

It is known \cite{FLP} that the toppling function of a state $\psi$ is the pointwise minimal non-negative function $F:G\to\ZZ_{\geq 0}$ satisfying $\psi(z)+\Delta F(z)\leq \val(z)$ pointwise (for all $z\in G\setminus S$) and $F_{|S}=0$.  The {\it Creutz identity} is defined as $\beta = \Delta {\bf 1}_S$ where ${\bf 1}_S$ is a characteristic function of $S$, \cite{creutz1991abelian}. Note that $\beta\geq 0$ on $G\setminus S$.

\begin{definition}
\label{def_0}
A sandpile state $\psi$ is called {\it recurrent} if $(\psi+\beta)^\circ = \psi$.
\end{definition}

It is known that the set of recurrent states for such a sandpile on $(G, S)$ is a finite group with respect to pointwise addition followed by a relaxation \cite{dhar1995algebraic}. This group, denoted by $C(G, S)$, is called the sandpile (or critical) group of $(G, S)$. Note that $\Delta$ is a homomorphism $\ZZ^G\to \ZZ^G$. Write  $\Gamma= G\setminus S$. Abusing notation we write $\Delta:\ZZ^\Gamma\to \ZZ^\Gamma$ using $\phi\in\ZZ^\Gamma\subset \ZZ^G$ (in the natiral way), $\Delta(\phi) = \Delta(\phi)|_{\Gamma}$.  It is known that $$C(G,S) = \ZZ^{\Gamma}/{\Delta \ZZ^{\Gamma}}.$$

In other words,  for $\psi,\phi\in \ZZ^\Gamma$ we say that $\psi$ and $\psi$ are equivalent, $\psi\sim \phi$, if there exists $F\in \ZZ^G$ such that $\phi=\psi+(\Delta F)_{|\Gamma}$. Then in each equivalence class $\ZZ^{\Gamma}/\sim$ there exists exactly one recurrent state, and the group structure on $C(G, S)$ coincides with that of  $\ZZ^{\Gamma}/\sim$,  see \cite{MR2246566} for further details.

\section{Sandpile group for infinite graphs}

In this section, we present definitions, statements, and explanations; proofs will follow in the next section.

{\bf Sandpile viewpoint.} Sandpile states can be defined for infinite graphs as well, but the relaxation issue is more subtle and usually is treated from the probabilistic perspective \cite{fey2009stabilizability, MR2031036, MR3350375,jarai2018sandpile}. However, the group structure can be defined for some infinite graphs \cite{heizmann2023sandpiles,kaiser2022abelian}. Let $G$ be any connected graph and $S$ be a subset of the vertices of $G$. A state $\phi\in \ZZ_{\geq 0}^G$ is called {\it relaxable} (stabilizable following \cite{fey2009stabilizability}) if there exists a bounded function $F\in \ZZ_{\geq 0}^{\Gamma}$ such that $\phi+\Delta F\leq \val$ pointwise. In this case, the toppling function is defined as the pointwise minimal function $F:\Gamma\to \ZZ_{\geq 0}$ such that $\psi+\Delta F\leq \val$ pointwise. We will define a family of graphs where the pointwise sum of two stable states is necessarily relaxable (which is not true for an arbitrary graph). 

{\bf Motivation: elusive group structure on recurrent sandpile states on $\ZZ^2$.} One can define a recurrent state on $\ZZ^2$ as a state whose restriction to each finite subset $G$ of $\ZZ^2$ (and letting $\ZZ^2\setminus G$ to be sinks) is a recurrent state. Unfortunately, the set of such states on $\ZZ^2$ does not form a group under pointwise addition and subsequent relaxation. On a finite graph, the recurrent states bijectively correspond to $S^1$-valued discrete harmonic functions on that graph (see below). Moreover, the set of $S^1$-valued discrete harmonic functions on $\ZZ^2$ is a group, but the relation between this group and the set of the recurrent sandpile states is unclear for infinite graphs. Meanwhile, there is a map (which is not related to the bijection in a finite case), preserving entropy, between these two objects \cite{schmidt2009abelian}, it is conjectured that this map is one-to-one almost surely. 

We aim to approach recurrent states on $\ZZ^2$ via recurrent states on infinite, slightly dissipative sandpiles. In this article, we define these slightly dissipative models with the hope that they will be helpful in the study of recurrent states on $\ZZ^2$. It would be a pity to lose the abelian group structure; we would like to know what can be saved.

Let $S\subset\ZZ^2$ and $\G=\ZZ^2\setminus S$. Let $C>0$. We call  $S$ a $C$-net if for each point $z\in\ZZ^2$ the edge-length of the shortest path from $z$ to $S$ is at most $C$.  We will consider subgraphs of $\ZZ^2$, but all our results can be easily extended for any other graph $G$ with a bounded valency and a $C$-net set $S$ of sinks.

Denote by $\ZZ^\G_B$ the abelian group of {\bf bounded} functions $\phi:G\to \ZZ$ such that $\phi(s)=0$ for all $s\in S$. Note that unlike the group of all integer-valued functions on $\ZZ^2$, which is not free,  $\ZZ^{\ZZ^2}_B$ is a free Abelian group \cite{nobeling1968verallgemeinerung}.

\begin{lemma}
\label{lemma_1}
If $S$ is a $C$-net, then any state $\phi\in \ZZ^\G_B$ is relaxable.
\end{lemma}

The Creutz identity is defined as above: $\beta = \Delta {\bf 1}_S$. Note that $\beta = \Delta({- \bf 1}_{\ZZ^2\setminus S})$.

\begin{definition}
A non-negative stable state $\phi\in \ZZ^\G_B$ is called {\it recurrent} if $(\phi+\beta)^\circ = \phi$.
\end{definition}

Then, it is not difficult to show that the set of all recurrent states in $\ZZ^2$ with a $C$-net sink set $S$ forms a group, which we denote by $C(\ZZ^2, S)$. For example, the neutral element in this group is computed as $(n\beta)^\circ$ for $n$ big enough (depending on $C$), and an explicit bound for $n$ can be given in terms of $C$. The inverse element for $\phi$ can be computed as $(-\phi+n\beta)^\circ$ for $n$ big enough. However, we provide another description of $C(\ZZ^2, S)$ to understand its finite subgroups better.

{\bf Algebraic viewpoint.} Our approach follows that of \cite{dhar1995algebraic, solomyak1998coincidence, schmidt2009abelian, lang2019sandpile}. Define the Laplace operator $\Delta: \ZZ^\G_B\to \ZZ^\G_B$ as follows. For $\phi\in \ZZ^\G_B$  and $z\in \G$ let
$$\Delta \phi(z)=\sum_{z'\sim z} \phi(z')-4\phi(z).$$  At $z\in S$ let $(\Delta \phi)(z)=0$ by definition. 

\begin{definition}
\label{def_1}
The {\it sandpile group} of $(\ZZ^2,S)$ is $$C(\ZZ^2,S)=\ZZ^\G_B/{\Delta \ZZ^\G_B}.$$
\end{definition}

Let $(S^1)^\G$ be the set of $S^1$-valued functions on $\G$ (as above we extend such functions to $S$ by zero). Note that the following sequence is exact: $$0\to \ZZ^\G_B\to \RR^\G_B \to (S^1)^\G\to 0.$$

The definition of the Laplacian operator can be naturally extended to the cases $\Delta: \RR^\G_B\to \RR^\G_B$ and $\Delta: (S^1)^\G\to (S^1)^\G$.
Denote by $\mathcal H_{S^1}(\G)$  the set of $S^1$-valued harmonic functions on $\G$: 

$$\mathcal H_{S^1}(\G) = \mathrm{ker} (\Delta: (S^1)^\G\to (S^1)^\G).$$ Note that such a harmonic function $\alpha\in \mathcal H_{S^1}(\G)$ is extended to $S$ by zero (so that we can compute $\Delta \alpha$ at neighbors of $S$), but for $z\in S$ it is not necessarily true that $\sum_{z'\sim z} \alpha(z')-4\alpha(z)=0$.

\begin{definition}
\label{def_group}
Let us draw the following diagram:

\begin{figure}[h]
\begin{tikzcd}
            & 0 \arrow[r] & \ZZ^\G_B \arrow[r, "\Delta"] \arrow[d] & \ZZ^\G_B \arrow[r] \arrow[d] & C(\ZZ^2,S) \arrow[r] & 0 \\
            & 0 \arrow[r] & \RR^\G_B \arrow[d] \arrow[r,"\Delta"]      & \RR^\G_B \arrow[r] \arrow[d] & 0           &   \\
0 \arrow[r] & \mathcal H_{S^1}(\G) \arrow[r] & (S^1)^\G \arrow[r, "\Delta"]                & (S^1)^\G \arrow[r]           & 0           &  
\end{tikzcd}
\label{figd}
\caption{Relation between the sandpile group and the $S^1$-valued harmonic functions on $\G$.}
\end{figure}

The first row defines the sandpile group $C(\ZZ^2,S)$, and the third row defines $\mathcal H_{S^1}(\G)$.
\end{definition}

\begin{theorem}
\label{thm_1}
If $S\subset\ZZ^2$ is a $C$-net, then the above diagram is commutative, and columns (extended by zeros on both sides) and rows are exact. In particular, $C(\ZZ^2,S) $ is isomorphic to $\mathcal H_{S^1}(\G)$. Moreover, classes of $C(\ZZ^2, S)$ (Definition~\ref{def_1})  bijectively correspond to  {\it recurrent} states (Definition~\ref{def_0}), i.e., the set of recurrent states of the sandpile on $(\ZZ^2, S)$ forms an abelian group.
\end{theorem}

\section{Proofs}
 
Define the following function $h\in \ZZ^\G_B$: 
\begin{equation}
\label{eq_1}
h(z) = \sum_{k=1}^{\mathrm{dist}(z,S)} 4^{C-k}.
\end{equation}

Note that $\Delta h(z)\leq 4^{C-\mathrm{dist}(z,S)}(-4+3(4^{-1}+1))\leq -1$ if $\mathrm{dist}(z,S)\geq 1$. If $\mathrm{dist}(z,S)=1$ then $\Delta\phi(z)\leq -4\cdot 4^C+3(4^C+4^{C-1})=-4^{C-1}\cdot 3\leq -1$. Note that $h\geq 0$ pointwise.

\begin{proof}[Proof of Lemma~\ref{lemma_1}] Consider a state $\phi\in\ZZ^\G_B$. Let $|\phi|\leq M$. Then $\phi+\Delta (M h)\leq \val$ pointwise and thus $\phi$ is relaxable.
\end{proof}
 
 Note that if  a stable state $\phi\geq 3$ everywhere then $\phi^\circ = (\phi^\circ+\beta)^\circ$. Indeed while relaxing $\phi+\beta$, we must do at least one toppling at every point of $\G$, and performing one toppling at every vertex we obtain  $$\phi+\beta+\Delta ({\bf 1}_\G) = \phi+\Delta (-{\bf 1}_\G)+\Delta ({\bf 1}_\G)=\phi.$$
 
Thus, it gives the following equivalent definition of a recurrent state.
\begin{definition}
A state $\phi\in\ZZ^\G$ is called recurrent if there exists $\psi\geq 3$ such that $\phi=\psi^\circ$.
\end{definition}

Note that $[\beta]=[0]$ in $C(\ZZ^2,S)$, see Definition~\ref{def_1}. Consider a state $\phi\in\ZZ^\G_B$. Let $|\phi|\leq M$. Then we can partially relax the state $\phi+ 2M\cdot 4^{C+1}\beta$ such that this partial relaxation is $\geq 3$ everywhere. So, for each $\phi\in\ZZ^\G_B$, an equivalent state exists, which is at least three everywhere. A relaxation of a state also does not change its equivalence class in $C(\ZZ^2, S)$. Thus we see that at least one recurrent state exists in each equivalence class of $\ZZ^\G_B/{\Delta \ZZ^\G_B}$. If two recurrent states $\phi,\psi$ belong to one equivalence class, then $\phi=\psi+\Delta F$ where $F$ is bounded. Then, as in the finite case, we look at the set $R=\{z| F(z) = \max F\}$ and see that while relaxing $\phi+\beta$ there will be no topplings on $R$ (in a finite case such $R$ is called a forbidden subconfiguration) which contradicts the definition of a recurrent state $\phi$.
 
\begin{proof}[Proof of Theorem~\ref{thm_1}] To prove this theorem by application of the snake lemma it is enough to show that a)$\mathrm{ker}(\Delta:\RR^\G_B\to \RR^\G_B)=0$ and b)$\mathrm{coker}(\Delta:\RR^\G_B\to \RR^\G_B)=0$.

Consider $\phi\in \RR^\G_B$ such that $\Delta\phi=0$. Note that $\Delta\phi=0$ in $\RR^\G_B$ does not imply that $\Delta\phi=0$ at all points since we have no restrictions for the values of $\Delta\phi$ on $S$. However, we know that $\phi$ is bounded. Take a point $z\in\ZZ^2$ where $\phi(z)$ is close to the $M=\mathrm{sup}\ |\phi|$. Then $\phi(z')$ is also close to $M$ for {\bf all} neighbors $z'$ of $z$. Note, however, that this is impossible for a neighbor of $S$ since $\phi|_{S} = 0$. But after no more than $C$ steps, we can come to $S$ from $z$, and thus, by choosing $\phi(z)$ close enough to $M$, we arrive at a contradiction, thus proving a).

To prove b), consider a function $\psi\in \RR^\G_B, |\psi(z)|\leq L, \forall z\in \G$. We want to show that there exists $\phi\in \RR^\G_B$ such that $\Delta \phi=\psi$. Since $\Delta$ is linear, then without loss of generality, we assume that $\psi\geq 0$.

Take $h$ from \eqref{eq_1}, then $Lh\geq 0$ satisfies $\Delta(Lh)\leq -L$. We will show that there exists $\phi, |\phi|\leq Lh$ such that $\Delta \phi=\psi$.
Construct such a function as the limit of the following sequence. Let $\phi_0=Lh$. Take all points $z\in\G$ where $\Delta \phi_0(z)\leq \psi(z)$ and diminish $\phi_0$ at these points, obtaining a function $\phi_1$, to have $\sum_{z'\sim z}\phi_0(z')-4\phi_1(z) = \psi(z)$. Then repeat and construct $$\phi_{n+1}(z) = \frac{1}{4}(-\psi(z)+\sum_{z'\sim z}\phi_n(z)).$$ 
Note that the condition $\Delta\phi_n\leq \psi$ is preserved, so  $\phi_{n+1}\leq \phi_n$. Also, the condition $\phi_n \geq -Lh$ is preserved. Hence, there exists a well-defined pointwise limit (pointwise infimum of $\phi_{n\to \infty}$) $\phi_{\infty}=\lim \phi_n$, satisfying $\Delta \phi_{\infty}=\psi$. 

\end{proof}

\begin{remark}
If $\G$ is not empty (i.e. $S\ne \ZZ^2$) then the sandpile group $C(\ZZ^2,S)$ is not trivial.
\end{remark}
One way to see that is to notice that there exists a (non-canonical) bijection between the elements of $C(\ZZ^2, S)$ and the set of spanning trees of $\ZZ^2$ with $S$ contracted into one point. One needs to carry the Dhar burning algorithm \cite{Dhar} exactly as in the case of finite graphs. Another way to see that  $C(\ZZ^2,S)$ is not trivial is to show that the state $\delta_v$ ($\delta_(v)=1$ and $0$ otherwise) is not in the image of $\Delta$ in Definition~\ref{def_group}.

In \cite{MR2031036} a sandpile model is called dissipative (Definition 2.2) is the supremum by $y$ of $\sum_{x}G(x,y)$ (the summation runs by all vertices of the graph, and $G(x,y)$ is the Green function) is finite. For the dissipative model, the sandpile group is defined, and its probabilistic properties are studied.

\section{Example: a ``ray'' graph}

Consider the following ``ray'' graph. Let $\Gamma=\{(i,0)\in\ZZ^2, i\geq 1\}$, and $S=\ZZ^2\setminus \Gamma$. This model is dissipative, a toppling at each $z\in\G$ results in losing two grains from the system (three when $z=(1,0)$).

In this case $C(\ZZ^2,S) = S^1$. Indeed, since $C(\ZZ^2,S) = \mathcal H_{S^1}(\G)$, we are looking for harmonic functions $a:\Gamma\to S^1$. Denote $a_i:=a(i,0)$ and let $a_0=0$. In this notation we are looking for harmonic sequences $0=a_0, a_1,a_2,a_3\dots \in S^1$, i.e. for all $n\geq 1$ 
\begin{equation}
\label{eq_2}
4a_{n}=a_{n-1}+a_{n+1} \pmod 1
\end{equation} since each vertex of $\G$ has four neighbors, and $a=0$ at two of them. So a choice of $a_1$ determines the sequence: $a_2 = 4a_1 \pmod 1, a_3=4a_2-a_1 = 15a_1 \pmod 1, a_4 = 4a_3-a_2 = 56a_1 \pmod 1$ etc. 

Hence, we conclude that $$a_n = c((2+\sqrt 3)^n-(2-\sqrt 3)^n) \pmod 1$$ for a certain constant $c\in [0,1)$. To construct the corresponding element in $C(\ZZ^2, S)$ we need to compute a sandpile state $\phi$ as the Laplacian of $a_n$ (treated as being in $\RR$, not in $S^1$), and then take the recurrent state $\gamma$ equivalent to $\phi$. 

So, let $a_i\in [0,1)\subset \RR$ and 
\begin{equation}
\label{eq_3}
\phi(n) = a_{n-1}+a_{n+1}-4a_n\in\RR.
\end{equation} Note that $\phi(n)\in\ZZ$ since $a_n$ is harmonic. However, this state $\phi$ may be not recurrent, indeed 
$$-4<a_{n-1}+a_{n+1}-4a_n<2, \text{\ hence \ } -3\leq \phi(n)\leq 1.$$ Note that the state $2+\delta_1$ is the neutral element of the sandpile group of $\G$.

$$e =  3 2 2 2 2 2 2 2 \dots = 2+\delta_1.$$

Indeed, $2+\delta_1$ is recurrent; it is the Creutz identity. Also one can check explicitly that $$(2\cdot(2+\delta_1))^\circ=2\cdot(2+\delta_1)+\Delta 1= 2+\delta_1.$$

Define a state $\gamma(n) = 3(2 + \delta_1) + \phi(n)$, we have $\gamma(n)\geq 3$, and then $\gamma^\circ$ represents the element of the sandpile group $C(\ZZ^2,S)$ corresponding to the element $$a_1=c \cdot 2\sqrt 3 \pmod1 \in S^1=\mathcal H_{S^1}(\G).$$

\begin{lemma}
There is a bijection between elements $g\in C(\ZZ^2,S)$ of order $k>1$ and harmonic functions $\phi\in  \ZZ_k^\G, \Delta \phi=0$.
\end{lemma}
\begin{proof}Since $kg=0$ then,  considering $g$ as a function $\G\to\RR$ we conclude that $kg(z)\in\ZZ$ for each $z\in\G$ and so we may define define $\phi(z) = k\cdot g(z)\in\ZZ_k$ which is clearly harmonic.
\end{proof}

\begin{remark}
\label{rem_torsion}
One can reformulate the above lemma as that the set $C(\ZZ^2,S)_k$ of the elements of $C(\ZZ^2,S)$ of order $k>1$ can be defined by the following short exact sequence:

\begin{center}
 \begin{tikzcd} 
0  \arrow[r]  & C(\ZZ^2,S)_k  \arrow[r] & \ZZ_k^{\G} \arrow[r, "\Delta"] & \ZZ_k^{\G}.\\
 \end{tikzcd}
\end{center}

\end{remark}

Let us see how it works for the ray graph $\G$. Each torsion subgroup of $S^1=\mathcal H_{S^1}(\G)$ is generated by $1/k, k\in \ZZ_{>1}$. For $a_1=1/k$, in the above notation, instead of a sequence $a_n \pmod 1$ we consider a sequence $$b_n\in \ZZ_k, b_n=k\cdot a_n \pmod k$$ with $b_0=0,b_1=1, b_{n+1} = 4b_n - b_{n-1}\in \ZZ_k$. Note that any such sequence is periodic, $b_N=0, b_{N+1}=1$ for a certain number $N$. So the corresponding sandpile state will be periodic with period $N$. However it will not be purely periodic, there can be a small initial part that is not included in the period (as in the Creutz identity), see the examples below.

Let $k=2$. Then, using \eqref{eq_2} we compute  the sequence $a_n$ and obtain $0,1/2, 0, 1/2,\dots$, then, by \eqref{eq_3} we see that $\phi(n)=-2$ if $n$ is odd and  $\phi(n)=1$ if $n$ is even. So $\gamma(n)=7, 7, 4, 7, 4, \dots$ and $\gamma^\circ= 1,3,0,3,0,3,0,\dots$ is the (unique) element of order two in $C(\ZZ^2,S)$.

Similarly, we compute the elements of order three:

$$2 1 2 3 3 2 1 1 2 3 3 2 1 1 2 3 3 \dots$$

$$1 0 3 1 1 2 3 3 2 1 1 2 3 3 2 1 1\dots $$

\begin{remark}
For $\G' = \{(i,0)\in\ZZ^2, i\in\ZZ\}, S=\ZZ^2\setminus \G'$ we obtain $C(\ZZ^2,S) = S^1\times S^1$ (we can choose $a_0,a_1$ arbitrarily, and this choice defines all the other elements $a_n, n\in\ZZ$),  the identity element is

$$\dots 2 2 2 2 2 2 2 2 \dots$$

The elements of order two are

$$\dots 3 0 3 0 3 \dots \text{\ and \ } \dots 0 3 0 3 0 \dots$$

For $\G'$ the elements of finite order are exactly periodic recurrent sandpile states.
\end{remark}

\begin{remark}
Note that in \cite{shirai2016solvable} the same $C(\ZZ^2, S)$ ($G'$ from the previous remark) appeared as the appropriate algebraic dynamic system to describe the sandpile infinite volume limit on a ladder graph $\ZZ\times\{1,2\}$. So one is tempted to consider intermediate sandpile models on $\ZZ\times\{1,2\}$ where the set $S$ of sinks is given by $S=n\ZZ\times \{2\}$. 
\end{remark}

\section{Example: no torsion}

Let us construct an example of a subset $S\subset \ZZ^2$ such that $C(\ZZ^2, S)$  is non-trivial and contains no elements of finite order (except the neutral element, of course). 

Consider the ``ray'' graph from the previous section. Recall that for each $k\in\ZZ_{\geq 2}$ there exists a number $d(k)$ such that for a harmonic function in $b:\ZZ_{>0}\to \ZZ_k$, with $b_0=0,b_1=1$, we have that $b_{d(k)}=0$ and if $b_1\ne 0$ then $b_m=0$ if and only if $d(k)| m$. On the other hand, $b_{d(k)-1}$ determines $b_1$, and if $b_{d(k)-1}=0$ then $b_1=0$.

The graph $G$ will be a line with attached intervals, i.e. $$\G=\{(i,0),i\in \ZZ_{\geq 1}\}\bigcup_{j=1}^\infty R(k_j,K_j)$$ where $R(k_j,K_j)=\{(k_j,K)\in\ZZ^2| 0\leq K\leq K_j\}$ the numbers $k_j,K_j$ will be chosen to kill all the torsions. Let $S=\ZZ^2\setminus \G$. 

Let us construct a graph with no nontrivial harmonic functions $b\in\ZZ_k^\G$. Let all $K_j=d(k)$ and $k_1=d(k)-1, k_{i+1}=k_i+(d(k)-1)$. Take a harmonic function $b\in \ZZ_k^\G$. 

Note that each element $b: G\to \ZZ_k$ of order $k$ in $C(\ZZ^2,S)$ is completely determined by $b(1,0)$ and all $b(k_j,K_j), j\in\ZZ_{\geq 1}$. Note that by the definition of $d(k)$ we have that $b(k_j,0)=0$ for all $j\in\ZZ_{\geq 1}$. Since $k_1=d(k)-1$, this implies that $b(1,0)$ also must be zero. Then, $b(k_2,0)=0$ implies that $b(k_1+1,0)=0$ and so $b(k_1,1)=0$ and therefore $b(k_1,K_1)=0$. Similarly, $b(k_j,K_j)=0$ for all $j\geq 2$, so $b\equiv 0$.

Adapting this argument, we will kill all the torsions of $C(\ZZ^2, S)$. Let $p_1=2,p_2=3,p_3=5,\dots$ the sequence of all prime natural numbers. Let $k_1=3, K_1=2$. This choice assures that there is no $2$-torsion in the part $x\leq k_1$ of $\G$. 

We proceed by induction and define $K_{j+1}=K_j\cdot d(p_{j+1})$ and $k_{j+1}=k_j+K_{j+1}\pm1$. The choice of sign is made in the following way. For all $p_1,p_2\dots, p_j$ we already know that on $G_{\leq k_{j+1}}=\{(x,y)\in \G | x\leq k_{j+1}\}$ there exists no nontrivial harmonic functions with values in $\ZZ_{p_j}$. 

A harmonic function with values in $\ZZ_{p_{j+1}}$ on $G_{\leq k_{j+1}}$ is uniquely determined by $b(1,0)\pmod {p_{j+1}}$ and all those $b(k_{i+i'},K_{i+i'})$ where $K_i|d(p_{j+1}) \pmod {p_{j+1}}$ and $i'\geq 0$. Since $b(k_{i+i'},0)=0\pmod {p_{j+1}},i'\geq 0$ we see that $b(k_{i}+x,0)=0 \pmod {p_{j+1}}$ for $x\geq 0$. Finally, by choosing plus or minus in $k_{j+1}=k_j+K_{j+1}\pm1$ we may assure that $b(1,0)=0\pmod {p_{j+1}}$.  So we constructed $S$ and proved the following lemma.

\begin{lemma} The group $C(\ZZ^2,S)$ is $S^1/{\sim}$ with new relations $1/2\sim 0,1/3\sim 0,\dots, 1/n\sim0,\dots$.
\end{lemma} 
This group is not trivial and is isomorphic to $\RR$ as an abelian group.

\section{Example: all torsions}

Fix $m,n\in \ZZ_{\geq 2}$. Let $S_{m,n}=\{(mi,nj), i,j\in\ZZ\}, \G_{m,n}=\ZZ^2\setminus S_{n,m}$. Note that $S_{n,m}$ is a $C$-net for $C=m+n$. 

\begin{remark}
One can directly see that $C(\ZZ^2,S_{2,2})$ contains $\ZZ_n$ for each natural $n$. Moreover, the corresponding $\ZZ_n$-valued harmonic function $\phi$ on $\G_{2,2}$ may be made periodic as $\phi(i,j)=\phi(i,j+2)$. 
\end{remark}

\begin{lemma} 
Fix any $m,n\in \ZZ_{\geq 2}$. Then $C(\ZZ^2,S_{n,m})$ contains $\ZZ_p$ for all primes $p$. 
\end{lemma}

We will construct a periodic harmonic function on $\G_{m,n}$ in $\ZZ_p$ for all primes $p$ big enough. To do that we construct a harmonic function $\phi$ on the infinite cylinder $W_n=(\ZZ^2, S)/{\sim}$ where we say that $(x,y) \sim (x,y+n)$, and then we lift such a harmonic function to $(\ZZ^2, S)$ in an obvious way. 

On $W_n$, assign any elements of $\ZZ_p$ to the vertices $$(0,i), 1\leq i\leq n-1, (-1,j), 1\leq j\leq n-1.$$ Denote this vector by $v_0\in V=(\ZZ_k)^{2n-2}$, $v$ represents  the values of $\phi$ at these vertices. Choose any value $x_0$ for $\phi(1,0)$. By harmonicity of $\phi$, the vector $v_0$ and $x_0$ uniquely define values of $\phi$ at vertices $(m,i), 0\leq i\leq n-1, (m-1,j), 0\leq j\leq n-1$, and we have one condition on $v_0$ and $x_0$ since $\phi(m,0)$ must be zero. Let us see what $\phi(m,0)$ would be if we start with $v_0=0,x_0=1$.  If it is not zero, we can choose any $v_0$, and then by tuning $x_0$, we can obtain $0$ for $\phi(m,0)$. Then we look at values $v_1$ composed of $\phi(m,i),\phi(m-1,i),1\leq i\leq n-1$ and then  by choosing an appropriate $x_1 = \phi(m+1,0)$ we can further extend $\phi$, etc. So, given $v_k$ composed of $\phi(km,i),\phi(km-1,i),1\leq i\leq n-1$ and by tuning $x_k=\phi(km+1,0)$ we can extend $\phi$ harmonically and determine $v_{k+1}$. Then we can choose two indices $k,k'$ such that $v_k=v_{k'}$ and repeat everything in between to get a harmonic function on the cylinder. 

If for $v_0=0$ and $x_0=1$ we have $\phi(m,0)=0$ then let us extend such $\phi$ further and see if $\phi(2m,0)$ is zero or not. If it is not zero, then we repeat the argument above, but for the fundamental domain twice bigger. If $\phi(2m,0)=0$ then we exten our $\phi$ further.


\section{Pontryagin duality perspective} Given a group $A$ one can consider the group $\hat A$ of all continuous homomorphisms $S\to S^1$. Note that if $A=\ZZ^{\oplus \G}$ (the direct sum, i.e. the set of $\ZZ$-valued functions on the vertices of $\G$ with finite support) equipped with the discrete topology, then $\hat A=(S^1)^\G$ (the infinite product of $S^1$ with the product topology, so $\hat A$ is a compact group).   Applying the Pontryagin duality to the lower row of Figure~\ref{figd} we get 

\begin{center}
 \begin{tikzcd} 
0 & \arrow[l]  \widehat{\mathcal  H_{S^1}(\G)} & \arrow[l]  \ZZ^{\oplus \G} &\arrow[l, "\Delta" swap]    \ZZ^{\oplus \G}.
 \end{tikzcd}
\end{center}

 So we can identify $\widehat{\mathcal  H_{S^1}(\G)}$ with the factor group of the Laurent polynomials $\ZZ[x,y]$ by the ideal generated by $4-x-x^{-1}-y-y^{-1}$ (this is the group of addition operators, cf. \cite{schmidt2009abelian,redig}). 
 
 Note that there is a natural map 
 $$\widehat{\mathcal  H_{S^1}(\G)}= \widehat{C(\ZZ^2,S)}\to C(\ZZ^2,S),$$
 
 which can be read from the diagram below:
 
 \begin{figure}[h]
\begin{tikzcd}
            0 \arrow[r] & \ZZ^{\oplus\G} \arrow[r, "\Delta"] \arrow[d] & \ZZ^{\oplus \G} \arrow[r] \arrow[d] & \widehat{C(\ZZ^2,S)} \arrow[r]\arrow[d] & 0 \\
            0 \arrow[r] & \ZZ^\G_B \arrow[d] \arrow[r,"\Delta"]      & \ZZ^\G_B \arrow[r] \arrow[d] \arrow[r] & C(\ZZ^2,S) \arrow[r] & 0   \\
            0 \arrow[r] & {\ZZ^\G_B/\ZZ^{\oplus\G}} \arrow[r,"\Delta"]   & {\ZZ^\G_B}/{\ZZ^{\oplus\G}} & & \\
\end{tikzcd}
\caption{The map from the dual of the sandpile group to the sandpile group. }
\label{figd}
\end{figure}

\begin{theorem} The map $\widehat{C(\ZZ^2,S)}\to C(\ZZ^2,S)$ is injective.
\end{theorem}
It is enough to prove that if $\Delta f$ has a finite support for $f\in \ZZ^\G_B$ then $f$ has compact support. Indeed, very far from the support of $\Delta f$ the function $f$ is harmonic. So, suppose it attains a local maximum  $M$ (which is the case since $f$ is bounded) at a vertex $v$. In that case, $f$ must be equal to $M$ at the neighbors of $v$, the neighbors of the neighbors of $v$ etc, but eventually, this set of the vertices reaches $S$, thus giving a contradiction since $f_{|S}=0$.

From the properties on Pontryagin duality, we know, for example, that the (discrete)  group $\widehat{C(\ZZ^2, S)}$ being torsion-free is equivalent to the (compact) group $C(\ZZ^2, S)$ being connected. It is possible that $\widehat{C(\ZZ^2,S)}$ is torsion free but $C(\ZZ^2,S)$ has torsion.
 
From the lower row of Figure~\ref{figd} and Remark~\ref{rem_torsion} we see that the $k$-torsion of $C(\ZZ^2,S)$ is the kernel of the map $\Delta: \ZZ_k^\G\to \ZZ_k^\G$. This map is injective if and only if the dual map $\Delta:  \ZZ_k^{\oplus\G}\to \ZZ_k^{\oplus \G}$ (for the functions with the finite support) is surjective. We summarize these facts in the following theorem.

\begin{theorem}
The following facts are equivalent to each other:

\begin{itemize}
\item $C(\ZZ^2,S)$ has no $k$-torsion,
\item $\Delta: \ZZ_k^\G\to \ZZ_k^\G$ is injective,
\item $\Delta: \ZZ_k^{\oplus \G}\to \ZZ_k^{\oplus\G}$ is surjective ($\ZZ_k^{\oplus \G}$ is the set of $\ZZ_k$-valued functions on $\G$ with finite support).
\end{itemize}

\end{theorem}

\section{Discussion and a large pile of sand}

Let us speculate how to get the Abelian group's structure on recurrent sandpile states on $\ZZ^2$. We start with the group $R' = \ZZ^\G_B/{\Delta \ZZ^\G_B}$ of bounded integer-valued functions on $\G=\ZZ^2$ modulo relations defined by the toppling operations.  A natural map exists from the set of all the states whose restriction to any finite subset of $\ZZ^2$ is a recurrent state to the group $R'$.

This map, however, is not injective: let $X\subset\ZZ^2$ be the set of points $(i,0), i\in \ZZ$. Consider the state $\phi$, which is equal to $2$ everywhere except $X$, where it is equal to $3$. Let $$f:\ZZ^2\to \ZZ, f|_X=1, f_{\ZZ^2\setminus X} = 0.$$ Then both states $\phi$ and $\psi=\phi+\Delta f$ have the property that the restriction of such a state to any finite subset of $\ZZ^2$ is recurrent. Note that $\psi$ contains no finite forbidden subconfiguration, but it contains an infinite forbidden subconfiguration on $X$. Indeed, the value of $\phi(v)=1$ at each vertex $v\in X$ is less than the internal (i.e., inside $X$) degree of $v$. So, we may say that we would like to consider only the set $R$ of states that do not contain neither finite nor infinite forbidden subconfiguration. 

We want to define a reasonable addition operation in $R$, which is made up of a pointwise addition followed by a relaxation. It seems inevitable that this operation should be based on the addition in $R'$. 

{\bf Harmonic $S^1$-valued functions and non-invertibility of $\Delta$.} The relation between $R'$ and the group of $S^1$-valued harmonic functions on $\ZZ^2$ is subtle because it is impossible to follow the reasoning in the proof of Theorem~\ref{thm_1}. Indeed, the map $\Delta: \RR_B^{\ZZ^2}\to \RR_B^{\ZZ^2}$ is not invertible. To the best of our knowledge, no good description (e.g., in the form of an explicit set of generators) of the group $\RR_B^{\ZZ^2}/{\Delta \RR_B^{\ZZ^2}}$ is available. Another strategy is possible: let $\sigma$ be the shift of $\ZZ^2$ by the vector $(1,0)$. Then it is known \cite{schmidt2009abelian} that for $\phi\in \ZZ_B^{\ZZ^2}$ the state $(\sigma-1)^3\phi$ belongs to the image of $\Delta$. Thus one can take $\phi\in R$, send it to  $(\sigma-1)^3\phi$, then lift in $\RR_B^{\ZZ^2}$ by $\Delta^{-1}$ and send it to $(S^1)^{\ZZ^2}$ using $\pmod 1$. This will produce an $S^1$-valued discrete harmonic function $\xi(\phi)$ on $\ZZ^2$. It is known  \cite{schmidt2009abelian} that $\xi$ is an equivariant, surjective, and entropy-preserving map. However, it is unknown whether $\xi$ is almost surely one-to-one.

{\bf Naive perspective: taking a quotient of $R'$.} While considering a sandpile state $\phi$ on an infinite graph, it is reasonable to look at the average number $A(\phi, X)$ of grains per vertex in large but finite parts $X$ of the graph. Note that $$A(\phi+\psi, X)= A(\phi, X)+A(\psi, X),$$ and the relaxation defined by a bounded toppling function does not change it, $A(\phi+\Delta F, X)\approx A(\phi, X)$ if $|F|<C$ is fixed and $X$ is big enough. Stable states $\phi$ satisfy $A(\phi,X)\leq 3$ for each $X$. The result of the pointwise addition of two such states may violate the latter condition, and a relaxation by a bounded toppling function may not help since any grain of sand cannot travel far away from its initial position.

A natural suggestion would be to factorize $R'$ by some element to correct the average number of grains. For example, we could declare that the state $\bf 1$, equal to one at each vertex of $\ZZ^2$, is doomed to be equivalent to the neutral element. This allows us to compute certain sums of elements, via reducing them to stable states by subtracting $\bf 1$. However, consider a state $\phi$ which is $2$ at vertices with positive ordinate and $3$ at vertices with non-positive ordinate. $\phi$ is recurrent but $4\cdot \phi$ cannot be relaxed even if we subtract $\bf 1$ up to three times. If we subtract $\bf 1$ four times, half of the vertices of $\ZZ^2$ have negative amount of grains, which will not change after a relaxation with a bounded toppling function.

Similarly, factorizing by an element $\phi\geq 0$ with an infinite support does not solve the problem that in a given state we can have a lot of sand in one part $X$ and no sand in another part $X'$, and these parts $X, X'$ are far from each other. Then, subtracting $\phi$, we can diminish the average amount of sand in $X$, which causes the average amount of sand in $X'$ to become negative. So, it would be better to factorize $R'$ by elements $\phi_i, i\in I$ with finite support. Then, this factorization does not help if there is no bound on the minimal distance between vertices of $\ZZ^2$ and $\phi_i, i\in I$ (otherwise we again need unbounded toppling functions).

Hence, the natural suggestion would be to factorize $R'$ by operations of adding a grain to a vertex of a certain $C$-net. A natural candidate is the set $S$ of points $(ni,nj), i,j\in \ZZ$ with a fixed $n$. Note that this is more than claiming $S$ to be the set of sinks, for adding $4$ grains to $(0,0)$ is the same as adding one grain to each neighbor of zero. That was the initial idea for this article. We proved above that if $S$ is a $C$-net in $\ZZ^2$ and we declare $S$ to be sinks, then a group $C(\ZZ^2, S)$ is defined. However, note that for $S\subset S'$ there is no natural homomorphism between $C(\ZZ^2, S)$ and $C(\ZZ^2, S')$ because adding new sinks does not respect the addition of recurrent states.
 
So, we concentrate on the above factorization idea. Define $\widetilde C(\ZZ^2,S)$  by considering the following diagram 

$$\ZZ_B^{\ZZ^2}\to \ZZ^{\ZZ^2\setminus S}_B \to \widetilde C(\ZZ^2,S)\to 0$$
where the second map is the composition of the discrete Laplacian and restriction.

Note that $ \widetilde C(\ZZ^2, S)$ is a factor group of $C(\ZZ^2, S)$, we factorize by elements $a_v$ which represent the addition of one grain to each neighbor of a given sink $v\in S$.  

{\bf The (elusive) sandpile group of $\ZZ^2$ as a projective limit.} Consider a nested system of infinite sets $S_0\supset S_1\supset S_2\supset \dots$ with $\bigcap S_i=0$. This system gives rise to an inverse system of $\widetilde C(\ZZ^2, S_i)$, so there exists a map from $R'$ to the projective limit of this system. 

{\bf Question.} What are the properties of the map $R'\to \varprojlim \widetilde C(\ZZ^2,S_i)$? Is it mono/epimorphism? Does it have any cohomological interpretation for a given nested family of $S_i$? Note that $$\varprojlim \ZZ^{\ZZ^2\setminus S}_B\ne \ZZ^{\ZZ^2}_B$$ since the element $\phi$ such that $\phi|_{\ZZ^2\setminus S_0}=0, \phi|_{S_{i+1}\setminus S_i} = i+1$ belongs to $\varprojlim \ZZ^{\ZZ^2\setminus S}_B$ but does not belong to $\ZZ^{\ZZ^2}_B$.

A direct limit of sandpile groups for certain finite subsets of $\ZZ^2$ is constructed in \cite{lang2019sandpile,lang2022sandpile}. Their idea is as follows: there exists a canonical monomorphism from the sandpile group of a rectangular $G_{m,n}$ with sides $m,n$ to the sandpile group of $G_{km+k-1,ln+l-1}$. Then we obtain a direct system of abelian groups by tiling $\ZZ^2$ by rectangular $G_{m,n}$. This can be contemplated on the level of $S^1$-valued harmonic functions on $G_{m,n}$. Given such a function, we copy (up to sign and reflection) it according to the tiling. On the sandpile level, this does not look very nice. In \cite{lang2019harmonic}, another projective limit is constructed, but again on the side of $S^1$-valued discrete harmonic functions, which cannot be translated to sandpile states since we cannot cross the ``non-invertibility of $\Delta$'' river. Contrary to their approach, our constructions is performed on the side of sandpile states.

{\bf Classical question: a big pile.} One of the classical problems in sandpile studies is to determine the shape of the boundary of the relaxation of a huge pile of sand at the origin \cite{PS, LPS,levine2013apollonian, levine2009strong, aleksanyan2019discrete, bou2020dynamic}. 

Let us see what happens if we put $1.2\cdot 10^7$ grains to the vertex $(3,3)$ on empty $\ZZ^2$ with periodic sinks at $(6i,6j),i,j\in\ZZ$. The boundary of the relaxed sandpile is close to a circle of radius $40$; see Figure~\ref{fig}.

The relaxation of  $10^{30}$ grains at $(3,3)$ contains an inner circle of radius $212.66$ and center $(3,3)$ and is contained in an outer circle of radius $216.56$. It seems the rescaled boundary of a relaxation of $n$ grain at one vertex on a plane with periodic sinks tends to a circle when $n\to\infty$, and that this may be proven more or less as in \cite{alevy2020limit} after averaging since after averaging a sandpile dynamics with periodic sinks looks like a sandpile with equally slightly dissipative vertices.

\begin{center}
\begin{figure}
\includegraphics[scale=0.4]{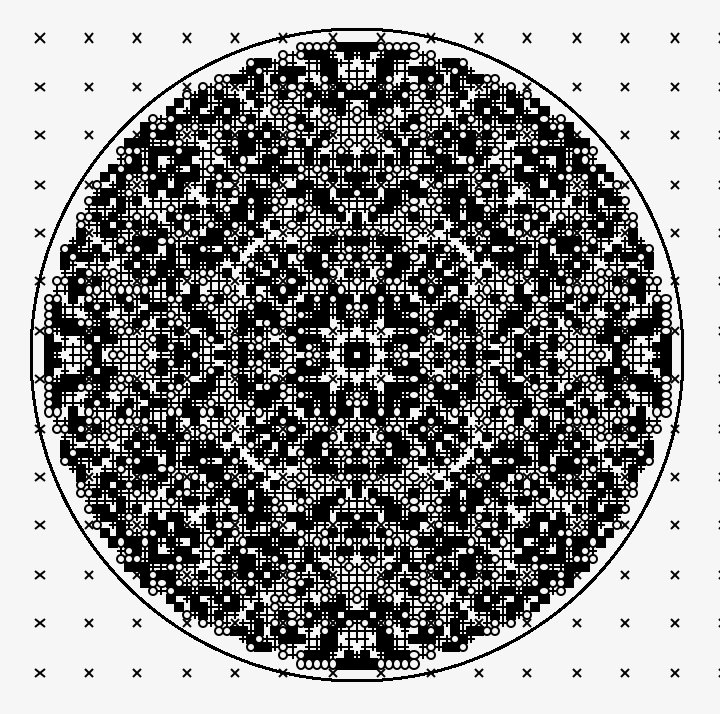}
\caption{The result of a relaxation of $1.2\cdot 10^7$ grains at $(3,3)$ on $\ZZ^2$ with sinks at $(6i,6j),i,j\in\ZZ$. Skew crosses are sinks; white represents zero grains, circles one grain, black squares two grains, and crosses three grains. A circle of radius $40$ with the center $(3,3)$ is drawn.}
\label{fig}
\end{figure}
\end{center}




%


\end{document}